\documentclass[12pt]{article}
\usepackage{amsmath, amssymb, amsfonts, amsthm, color}
\usepackage{hyperref,verbatim}
\usepackage[margin=1cm]{caption}
\usepackage{tikz}

\usetikzlibrary{arrows,shapes,snakes,automata,backgrounds,petri,patterns}

\newtheorem{proposition}{Proposition}[section]
\newtheorem{lemma}[proposition]{Lemma}
\newtheorem{corollary}[proposition]{Corollary}
\newtheorem{theorem}[proposition]{Theorem}
\newtheorem{claim}{Claim}

\newtheorem{observation}[proposition]{Observation}

\newtheorem{lem}[proposition]{Lemma}

\newtheorem{thm}[proposition]{Theorem}
\newtheorem{conj}[proposition]{Conjecture}

\newcommand\floor[1]{\left\lfloor#1\right\rfloor}
\newcommand\card[1]{\left|#1\right|}
\def\chil{\chi_{\ell}}
\def\chip{\chi_{p}}

\DeclareMathOperator{\mad}{mad}
\DeclareMathOperator{\ad}{ad}

\newcommand\blfootnote[1]{%
  \begingroup
  \renewcommand\thefootnote{}\footnote{#1}%
  \addtocounter{footnote}{-1}%
  \endgroup
}

\begin{document}
\font\smallrm=cmr8

\phantom{a}\vskip .25in
\centerline{{\large \bf  Planar Graphs of Girth at least Five are Square $(\Delta + 2)$-Choosable}}
\vskip.4in
\centerline{{\bf Marthe Bonamy}%
\footnote{\texttt{marthe.bonamy@u-bordeaux.fr}. 
This author is supported by the ANR Grant EGOS (2012-2015) 12~JS02~002~01.}}
\smallskip
\centerline{CNRS, LaBRI, Universit\'e de Bordeaux}
\centerline{Bordeaux, France}
\medskip
\centerline{and}

\centerline{{\bf Daniel W. Cranston}%
\footnote{\texttt{dcranston@vcu.edu}.
This author is partially supported by NSA Grant H98230-15-1-0013.}}
\smallskip
\centerline{Department of Mathematics and Applied Mathematics}
\centerline{Virgina Commonwealth University}
\centerline{Richmond, Virginia, USA, 23284}
\medskip
\centerline{and}

\medskip
\centerline{{\bf Luke Postle}%
\footnote{\texttt{lpostle@uwaterloo.ca}.
This author is partially supported by NSERC Discovery Grant No. 2014-06162.}} 
\blfootnote{\copyright~$\langle2018\rangle$. This manuscript version is made
available under the CC-BY-NC-ND 4.0 license
\url{http://creativecommons.org/licenses/by-nc-nd/4.0/}}
\smallskip
\centerline{Department of Combinatorics and Optimization}
\centerline{University of Waterloo}
\centerline{Waterloo, ON, Canada, N2l 3G1}

\vskip 1in \centerline{\bf ABSTRACT}
\bigskip

{
\noindent
We prove a conjecture of Dvo\v{r}\'ak, Kr\'al, Nejedl\'y, and \v{S}krekovski
that planar graphs of girth at least five with maximum degree $\Delta$ are
square ($\Delta+2$)-colorable for large enough $\Delta$.  In fact, we prove the
stronger statement that such graphs are square ($\Delta+2$)-choosable and even
square ($\Delta+2$)-paintable.
}

\vfill \baselineskip 11pt \noindent \date.
\vfil\eject
\baselineskip 18pt

\section{Introduction}
\label{sec:intro}
Graph coloring is a central area of research in discrete mathematics.
Historically, much work has focused on coloring planar graphs, particularly in
an effort to prove the 4 Color Theorem.  Since its proof
in 1976, research has expanded to numerous related problems.  One that has
received significant attention is coloring the \emph{square} $G^2$ of a planar
graph $G$, where $V(G^2)=V(G)$ and $uv\in E(G^2)$ if $\textrm{dist}_G(u,v)\le 2$.
Wegner~\cite{wegner77} conjectured that every planar graph $G$ with maximum degree
$\Delta\ge 8$ satisfies $\chi(G^2)\le\floor{\frac{3\Delta}2}+1$. He also
constructed graphs showing that this number of colors may be needed (his
construction is a minor variation on that shown in
Figure~\ref{fig:girth4DeltaC}, which requires $\floor{\frac{3\Delta}2}$ colors).  The
\emph{girth} of a graph $G$, denoted $g(G)$, is the length of its shortest cycle.
Since Wegner's construction contains many 4-cycles, it is natural to ask about
coloring the square of a planar graph $G$ with girth at least 5.
First, we need a few more definitions.

A \emph{list assignment} $L$ for a graph $G$ assigns to each vertex $v\in V(G)$ a list
of allowable colors $L(v)$.  A proper \emph{$L$-coloring} $\varphi$ is a proper vertex
coloring of $G$ such that $\varphi(v)\in L(v)$ for each $v\in V(G)$.  A graph $G$ is
\emph{$k$-choosable} if $G$ has a proper $L$-coloring from each list assignment $L$
with $|L(v)|= k$ for each $v\in V(G)$.  The \emph{list chromatic number}
$\chil(G)$ is the minimum $k$ such that $G$ is $k$-choosable.
Finally, a graph $G$ is \emph{square $k$-choosable} if $G^2$ is $k$-choosable.

\begin{conj}[Wang and Lih~\cite{WangL03}]
For every $k\ge 5$ there exists $\Delta_k$ such that if $G$ is a planar graph
with girth at least $k$ and $\Delta\ge \Delta_k$, then $\chi(G^2)=\Delta+1$.
\label{conjecture1}
\end{conj}

Borodin et~al.~\cite{BorodinGINT04} proved the Wang--Lih Conjecture for $k\ge
7$.  Specifically, they showed that $\chi(G^2)=\Delta+1$
whenever $G$ is a planar graph with girth at least $7$ and $\Delta\ge 30$.  In
contrast, for each integer $D$ at least 2, they constructed a planar
graph $G_D$ with girth 6 and $\Delta=D$ such that $\chi(G_D^2)\ge \Delta+2$.

In 2008, Dvo\v{r}\'ak et~al.~\cite{DvorakKNS08} showed that 
for $k=6$ the Wang--Lih Conjecture fails only by 1.
More precisely, let $G$ be a planar graph with girth at least $6$.
They showed that if $\Delta\ge8821$, then $\chi(G^2)\le \Delta+2$.
Borodin and Ivanova strengthened this result: in 2009 they
showed~\cite{BorodinI09girth6} that 
$\Delta\ge 18$ implies
$\chi(G^2)\le\Delta+2$ (and also~\cite{BorodinI09girth6choice} that
$\Delta\ge 36$ implies $\chil(G^2)\le \Delta+2$).
Dvorak et~al. conjectured that a similar result holds for girth 5.

\begin{conj}
[Dvo\v{r}\'ak, Kr\'al, Nejedl\'y, and \v{S}krekovski~\cite{DvorakKNS08}]
There exists $\Delta_0$ such that if $G$ is a planar graph
with girth at least $5$ and $\Delta(G)\ge \Delta_0$ then
$\chi(G^2)\le\Delta(G)+2$.
\label{conjecture2}
\end{conj}

Our main result verifies Conjecture~\ref{conjecture2}, even for
list-coloring.  Further, in Section~\ref{paint}, we extend the result to
paintability, also called online list-coloring.  Using a version of Lemma~3.21
in~\cite{CCP}, we can also extend the result to correspondence coloring.


\begin{theorem}\label{th:main}
There exists $\Delta_0$ such that if $G$ is a planar graph with girth at least
five and $\Delta(G)\ge \Delta_0$, then $G$ is square ($\Delta(G)+2$)-choosable.
In particular, we can let $\Delta_0=1,730^2+1=2,992,901$.
\end{theorem}

The number of colors in Theorem~\ref{th:main} is optimal, as shown by the
family of graphs introduced in~\cite{BorodinGINT04} and depicted in
Figure~\ref{fig:girth5Delta1}. The vertex $u$ and its $p$ neighbors
together require $p+1$ colors; since $v$ is at distance $2$ from each of them,
in total we need $p+2$ distinct colors.

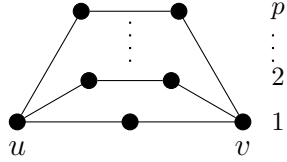
\begin{figure}[!h]
\center
\begin{tikzpicture}
    \tikzstyle{whitenode}=[draw,circle,fill=white,minimum size=9pt,inner sep=0pt]
    \tikzstyle{blacknode}=[draw,circle,fill=black,minimum size=6pt,inner sep=0pt]
\draw (0,0) node[blacknode] (a) [label=-90:$u$] {} 
-- ++(0:1.5cm) node[blacknode] (b) {}
-- ++(0:1.5cm) node[blacknode] (c) [label=-90:$v$] {};
\draw (a) -- ++(30:1.1cm) node[blacknode] (d) {};
\draw (a) -- ++(60:1.7cm) node[blacknode] (f) {};
\draw (c) -- ++(150:1.1cm) node[blacknode] (e) {};
\draw (c) -- ++(120:1.7cm) node[blacknode] (g) {};
\draw (d) edge  node {} (e); 
\draw (f) edge  node {} (g);
\draw[loosely dotted][thick] (1.5,0.8) edge  node {} (1.5,1.4);
\draw (3.2,0) node[right=1pt] (g) {\footnotesize $1$};
\draw (3.2,0.6) node[right=1pt] (g) {\footnotesize $2$};
\draw (3.2,1.4722) node[right=1pt] (g) {\footnotesize $p$};
\draw[loosely dotted][thick] (3.4,0.8) edge  node {} (3.4,1.3);
\end{tikzpicture}
\caption{A graph $G_p$ with girth 5, $\Delta(G_p) = p$, and $\chi^2(G_p)=\Delta(G_p)+2$.}
\label{fig:girth5Delta1}
\end{figure}

The girth assumption is tight as well, due to a construction directly inspired
from Shannon's triangle (see Figure~\ref{fig:girth4DeltaC}).  When coloring the
square, all $3p$ degree 2 vertices need distinct colors, since each pair has a
common neighbor.

\begin{figure}[!h]
\center
\begin{tikzpicture}[xscale=2.3,yscale=-2.3]
    \tikzstyle{whitenode}=[draw,circle,fill=white,minimum size=9pt,inner sep=0pt]
    \tikzstyle{blacknode}=[draw,circle,fill=black,minimum size=6pt,inner sep=0pt]
 \draw (0,0) node[blacknode] (a) [label=left:$u$] {}
 ++(0:1cm) node[blacknode] (b) [label=right:$v$] {}
 ++(120:1cm) node[blacknode] (c) [label=-90:$w$] {};


\draw (a) -- ++(-25:0.5517cm) node[blacknode] (ab) [label=-90:\small{$u_1$}] {};
\draw (a) -- ++(-50:0.7779cm) node[blacknode] (abn) [label=90:\small{$u_p$}] {};
\draw (b) edge node {}  (ab);
\draw (b) edge node {}  (abn);
\draw (b) -- ++(95:0.5517cm) node[blacknode] (bc) [label=150:\small{$v_1$}] {};
\draw (b) -- ++(70:0.7779cm) node[blacknode] (bcn) [label=-30:\small{$v_p$}] {};
\draw (c) edge node {}  (bc); 
\draw (c) edge node {}  (bcn);
\draw (a) -- ++(85:0.5517cm) node[blacknode] (ac) [label=30:\small{$w_1$}] {};
\draw (a) -- ++(110:0.7779cm) node[blacknode] (acn) [label=210:\small{$w_p$}] {};
\draw (c) edge node {}  (ac); 
\draw (c) edge node {}  (acn);
\draw[loosely dotted][thick] (ab) edge  node {} (abn);
\draw[loosely dotted][thick] (bc) edge  node {} (bcn); 
\draw[loosely dotted][thick] (ac) edge  node {} (acn);


\end{tikzpicture}
\caption{A graph $G_p$ with girth 4, $\Delta(G_p) = 2p$, and $\chi^2(G_p)=3p$.}
\label{fig:girth4DeltaC}
\end{figure}
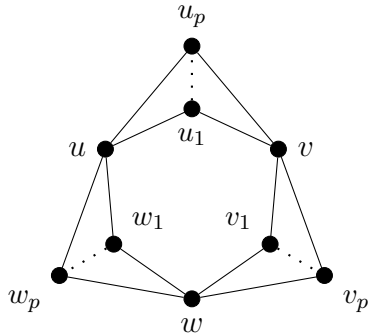

Theorem~\ref{th:main} is also optimal in another sense.
But before we can explain it, we must
introduce a more refined measure of a graph's sparsity: its maximum average
degree. The average degree of a graph $G$, denoted $\ad(G)$, is $\frac{\sum_{v
\in V}d(v)}{|V|}=\frac{2|E|}{|V|}$. The \emph{maximum average degree} of 
$G$, denoted $\mad(G)$, is the maximum of $\ad(H)$ over every subgraph $H$ of
$G$.  For planar graphs, Euler's formula links girth and maximum average degree.

\begin{lemma}[Folklore]\label{lem:euler}
For every planar graph $G$, $(\mad(G)-2)(g(G)-2)<4$.
\end{lemma}

Note that every planar graph $G$ with $g(G)\geq 7$ satisfies $\mad(G)<
\frac{14}5$. It was proved~\cite{blp14YEUJC} that
Conjecture~\ref{conjecture1} is true not only for planar graphs with $g \geq
7$, but also for all graphs with $\mad < \frac{14}5$ (in fact even for all
graphs with $\mad < 3- \epsilon$, for any fixed $\epsilon>0$). 
The theorem mentioned above of Borodin and Ivanova~\cite{BorodinI09girth6} for planar
graphs with girth at least 6 was strengthened~\cite{blp13DM} in the setting of
maximum average degree: every graph $G$ with $\mad(G) <3$ and $\Delta(G) \geq
17$ satisfies $\chi^2_\ell(G) \leq \Delta(G)+2$.
These results suggested that perhaps sparsity was
the single decisive characteristic when square list coloring planar graphs of high
girth.  However, as we show below, Theorem~\ref{th:main} cannot be strengthened
to require only $\mad(G)<\frac{10}3$ (rather than planar, with girth 5). 

Charpentier~\cite{c12} generalized the family of graphs shown in
Figure~\ref{fig:girth5Delta1} to obtain for each $C \in \mathbb{Z}^+$ a family
of graphs with maximum average degree less than $\frac{4C+2}{C+1}$, with
unbounded maximum degree, and whose squares have chromatic number $\Delta+C+1$.
For $C=2$ the construction, shown in Figure~\ref{fig:mad4i}, yields a family of
graphs with arbitrarily large maximum degree, maximum average degree less than
$\frac{10}3$, and whose squares are not $(\Delta+2)$-colorable.
In the square, all $p+4$ vertices $u, v_1, \ldots, v_p, w_1, w_2, x$
must receive distinct colors, since they are pairwise adjacent.
The maximum average degree of $G_p$ is reached on the graph $G_p$ itself. 
We can argue by induction that $\mad(G_p) < \frac{10}3$: note that $\mad(G_0)=\frac32<\frac{10}3$ and 
that $G_{p+1}$ is built from $G_p$ by adding precisely $3$ vertices and $5$ edges (if $\frac{2(a+5)}{b+3}\geq \frac{10}3$ then necessarily $\frac{2a}{b}\geq \frac{10}3$).


\begin{figure}[!h]
\center
\begin{tikzpicture}[scale=2.5,rotate=180]
\tikzstyle{blacknode}=[draw,circle,fill=black,minimum size=6pt,inner sep=0pt]
\draw (0,0) node[blacknode] (u) [label=right:$u$] {}
-- ++(-45:1cm) node[blacknode] (v1) [label=90:$v_1$] {}
-- ++(0:1cm) node[blacknode] (v1w1) {}
-- ++(0:1cm) node[blacknode] (w1) [label=90:$w_1$] {}
-- ++(45:1cm) node[blacknode] (x) [label=left:$x$] {};
 

\draw (u) 
-- ++(45:1cm) node[blacknode] (vp) [label=-90:$v_p$] {}
-- ++(0:1cm) node[blacknode] (vpwC) {}
-- ++(0:1cm) node[blacknode] (wC) [label=-90:$w_2$] {};
\draw (wC) edge  node {} (x);

\draw (u) -- ++(-20:0.75cm) node[blacknode] (v2) [label=90:$v_2$] {};
\draw (vp) -- ++(-38:0.7cm) node[blacknode] (vpw1) {};
\draw (vpw1) edge  node {} (w1);
\draw (wC) -- ++(180+28:0.7cm) node[blacknode] (v2wC) {};
\draw (v2wC) edge  node {} (v2);
\draw (v2) -- ++(-12:0.7cm) node[blacknode] (v2w1) {};
\draw (v2w1) edge  node {} (w1);
\draw (v1) -- ++(35:1.6cm) node[blacknode] (v1wC) {};
\draw (v1wC) edge  node {} (wC);

\draw[loosely dotted][very thick] (0.7,-0.1) edge  node {} (0.7,0.5);

\draw plot [smooth, tension=0.5] coordinates {(x) (2.7,1) (0.7,1) (u)};
\end{tikzpicture}
\caption{A graph $G_p$ with $\Delta(G_p) = p+1$, $\mad(G_p)=\frac{10p+6}{3p+4}< \frac{10}3$ and $\chi^2(G_p)=p+4$.
}
\label{fig:mad4i}
\end{figure}

\section{Definitions and notation}
\label{sec:defs}

Most of our definitions and notation are standard; for reference, though, we
collect them below.  Let $G$ be a multigraph with no loops.  The
\emph{neighborhood} of a vertex $v$ in $G$, denoted $N(v)$, is the set of neighbors
of $v$, i.e., $N(v)=\{u: uv\in E(G)\}$.  The \emph{closed neighborhood}, denoted
$N[v]$, is defined by $N[v] = N(v)\cup\{v\}$.  For a vertex set $S$, let $N(S) =
\cup_{v\in S}N(v)$ and $N[S] = \cup_{v\in S}N[v]$.
For a multigraph $G$ or digraph $D$ and a subset $S$ of its vertices, $G[S]$ or
$D[S]$ is the subgraph induced by $S$.  
The degree of a vertex $v$ in a multigraph $G$ is the number of incident edges.
The degree of a vertex $v$ is denoted $d_G(v)$, or $d(v)$ for short.
So, in particular, we may have $d_G(v) > |N_G(v)|$.  If $H$ is a subgraph of
$G$, then $d_H(v)$ is the number of edges of $H$ incident to $v$.  
For vertices $u$ and $v$ with $u\in N(v)$, the \emph{multiplicity} of the edge
$uv$ is the number of edges with $u$ and $v$ as their two endpoints.
The maximum and minimum degrees of $G$ 
are $\Delta(G)$ and $\delta(G)$, respectively; when $G$ is clear from context,
we may write simply $\Delta$ or $\delta$.  The \emph{girth} of $G$ is the
length of its shortest cycle.  

A multigraph $G$ is \emph{planar} if it can be drawn in the
plane with no crossings.  A \emph{plane map} is a planar embedding of a planar
multigraph such that each face has length at least 3.  
An \emph{$\ell$-face} (resp.~\emph{$\ell^+$-face}) is a face with boundary walk
of length equal to (resp.~at least) $\ell$.
The \emph{underlying map}
$G'$ of a plane embedding of a planar multigraph $G$ is formed from the
embedding of $G$ by suppressing all of its 2-faces (that is, repeatedly
identifying the two boundary edges of some remaining 2-face, until no 2-face
remains).  Suppose that $d_G(v)=2$ and that $N(v)=\{u,w\}$.  To
\emph{suppress $v$}, delete $v$ and add an edge $uw$.

\section{Proof of Main Theorem}\label{sec:proof}

Let $\Delta_0=1,730^2+1=2,992,901$ and let $k \geq \Delta_0$. To prove
Theorem~\ref{th:main},
it suffices to prove that every plane graph $G$ of girth at least five with
$\Delta(G) \leq k$ is square $(k+2)$-choosable.  Assume, for a contradiction,
that this does not hold, and consider a counterexample $G=(V,E)$ that
minimizes $\card{V}+\card{E}$.  We fix an embedding of $G$.
Let $L$ be a list assignment of $(k+2)$ colors to each vertex
of $G$ such that $G^2$ has no $L$-coloring. 
We reach a contradiction, by showing that $G$ must contain
some subgraph $H$ such that every $L$-coloring of $(G\setminus E(H))^2$ can be
extended to an $L$-coloring of $G^2$; such an $H$ is \emph{reducible}.  An
unusual feature of our proof is that we do not use the discharging method.
Instead, we use only the fact that every planar map has a vertex of degree at
most 5.

A vertex $v$ of $G$ is \emph{big} if $\deg(v)\ge \sqrt{k}$, 
and otherwise $v$ is \emph{small}.
The sets of big and small vertices of $G$ are denoted,
respectively, by $B$ and $S$.
Further, let $S_i = \{v\in S: |N_G(v) \cap B|=i\}$, i.e., small vertices
with exactly $i$ big neighbors. 
By Lemma 1.4, $\mad(G)<\frac{10}3$, so $|E(G)|<5|V(G)|/3$.  Thus, only a tiny
fraction of $V(G)$ can be big vertices.  Likewise, by planarity $\bigcup_{i\ge
3}S_i$ has size linear in the number of big vertices, again a very small
fraction of $|V(G)|$.  Hence, the vast majority of $V(G)$ is the subset
$\bigcup_{i=0}^2S_i$.  We show that $S_0$, the set of small vertices with only
small neighbors, induces an independent set.
Thus, we can decompose the planar embedding into regions, each defined by a
pair of big vertices.  We prove that any region with many vertices is
reducible.  To complete the proof, we show that some big vertex $v$ is adjacent
to few regions (here we use that every planar map has a vertex of degree at
most 5), so $v$ must be adjacent to a region with many vertices.

We begin with a few simple observations about $G$.

\begin{lem}\label{prop:connected}
Graph $G$ is connected and $\delta(G)\ge 2$.
\end{lem}

\begin{proof}
If $G$ is not connected, then one of its components is a smaller
counterexample, contradicting the minimality of $G$. 

If $G$ contains a vertex $u$ of degree $1$, color $G \setminus \{u\}$ by
minimality, and extend the coloring to $G$ as follows. Vertex $u$ has
exactly one neighbor $v$, whose degree is at most $k$, by assumption on $G$.
So, $|N_{G^2}(u)| = |N(v)\setminus \{u\}|+|\{v\}| \leq k$. 
Since $|L(u)|\ge k+2$, some color $c$ in $L(u)$
is available to use on $u$.  Coloring $u$ with $c$ gives an $L$-coloring
of $G^2$, a contradiction.
\end{proof}

A key observation is the next lemma, which shows that in a minimal
counterexample at least one endpoint of every edge is either big or adjacent to
a big vertex.

\begin{lemma}\label{lem:bigedge}
For every edge $uv$ of $G$, either $u\in N[B]$ or $v\in N[B]$.
\end{lemma}
\begin{proof}
Assume, for a contradiction, that there is an edge $uv \in E(G)$ such that $u , v
\not\in N[B]$. In other words, $u$ and $v$ and all their neighbors have
degree at most $\sqrt{k}$. By the minimality of $G$, there exists an
$L$-coloring $\varphi$ of $(G-uv)^2$. Now we recolor both
$u$ and $v$ to obtain an $L$-coloring of $G^2$. Since $u$
has at most $\sqrt{k}$ neighbors in $G$, each of which has degree at most
$\sqrt{k}$, vertex $u$ has
at most $k$ neighbors in $G^2$. Thus, at most $k$ colors appear on $N_{G^2}(u)$.
Since $|L(u)|=k+2$, at least two colors remain available for $u$, counting its
own color. Similarly, $v$ has at least two available
colors. Thus, we may extend $\varphi$ to $u$ and $v$ to obtain a proper
$L$-coloring of $G^2$, a contradiction.
\end{proof}

The next lemma extends Lemma~\ref{lem:bigedge}, by showing that if both endpoints of
an edge have degree two then both endpoints are adjacent to big vertices.
 
\begin{lemma}\label{lem:degtwo}
If $u$ and $v$ are adjacent vertices of degree two, then $u\in N(B)$ and $v\in N(B)$.
\end{lemma}
\begin{proof}
Suppose not. Let $u$ and $v$ be adjacent vertices of degree $2$ such that
$v\not\in N(B)$. Let $w$ be the neighbor of $v$ distinct from $u$. By the
minimality of $G$, there exists an $L$-coloring $\varphi$ of $(G \setminus
\{u,v\})^2$.  Since $u$ has at most $k+1$ neighbors in $G^2$ that are already
colored, we can extend $\varphi$ to $u$. Now $|N_{G^2}(v)| = |N(w) \setminus
\{v\}|+|\{w,u\}|+|N(u) \setminus \{v\}| \leq d(w)+2$. Since $w\not\in B$, $d(w)
< \sqrt{k}$, so we can extend $\varphi$ to $v$, which yields an
$L$-coloring of $G^2$, a contradiction. 
\end{proof}

The intuition behind much of the proof is that small vertices with only small
neighbors can always be colored last.  A key ingredient in formalizing
this intuition is a new plane multigraph. 
Let $G'$ denote the plane multigraph obtained from $G$ by first suppressing
vertices of degree $2$ in $S\setminus N(B)$ (defined at the end of
Section~\ref{sec:defs}) and then contracting each edge with one endpoint in
each of $S_1$ and $B$. Note that there is a natural bijection between the faces
of $G'$ and those of $G$.  We will use Lemmas~\ref{lem:bigedge}
and~\ref{lem:degtwo} to prove structural properties of $G'$. 

Since big vertices of $G$ are not identified with each other in the
construction of $G'$, we also let $B$ denote the vertices of $G'$ that contain
a big vertex of $G$. Let $S' = V(G')\setminus B$.
Note that neither suppressing nor contracting decreases the degree of a
vertex in $B$; thus, we conclude the following.

\begin{observation}\label{obs:bigdegree}
For every vertex $v$ in $B$, we have $d_{G'}(v)\geq d_G(v)$.
\end{observation}

Let $G''$ denote the underlying map of $G'$. We will next show that there is a big
vertex (in $G$ and $G'$) whose degree in $G''$ is small; in other words, $v$ has
many edges in $G'$ to the same neighbor.  But first we need the following general
lemma about plane maps with certain properties, the hypotheses of which (as we will 
show) are satisfied by $G''$.

\begin{lem}
\label{prop:40}
Let $H$ be a plane map and $A$, $C$, and $D$ be disjoint vertex sets such that
$V(H)=A\cup C \cup D$, every $v\in C$ satisfies $|N(v)\cap D|\ge 2$, and for
all $v_1,v_2\in C$ such that $|N(v_1)\cap D|=|N(v_2)\cap D|=2$, it holds that
$N(v_1)\cap D\ne N(v_2)\cap D$. If $A$ is an independent set and $d(v)\ge 3$
for all $v\in A$, then there exists $u\in D$ with 
$d_{H[D]}(u) \leq 10$ and $d_H(u)\le 40$.
\end{lem}

Before proving this lemma, we show how we apply it.

\begin{lemma}
\label{lem:smallbigdeg}
\label{cor:highmultiplicity}
There exists $v\in B$ with $d_{G''[B]}(v)\le 10$ and $d_{G''}(v)\le 40$. 
Further, there exists $u \in V(G')$ (recall that
$V(G'')=V(G')$) such that at least $\frac{\sqrt{k}}{40} - 1$ consecutive faces
of length two in $G'$ have boundary $(u,v)$.
\end{lemma}
\begin{proof}
To prove the first statement, we apply Lemma~\ref{prop:40} to $G''$ with
$D=B$, $C=\bigcup_{i\ge 2} S_i$, and $A=\{v\in S_0: d(v)\ge 3\}$.
So we must show that this application satisfies the necessary hypotheses.
(Recall that when forming $G'$, we suppressed all $w\in S_0$ with $d(w)=2$ and
we contracted into $B$ all $w\in S_1$.)
Note that $A$ is an independent set, by Lemma~\ref{lem:bigedge}.
Now suppose there exist $v_1,v_2\in C$ such that $|N_{G''}(v_1)\cap
B|=|N_{G''}(v_2)\cap B|=2$ and $N_{G''}(v_1)\cap B=N_{G''}(v_2)\cap B$; say
$N_{G''}(v_1)\cap B = \{b_1,b_2\}$.  When $G''$ was formed from $G$, vertices
$v_1$ and $v_2$ may have gained neighbors in $B$, but they did not lose
neighbors.  Since $v_1,v_2\in \bigcup_{i\ge 2}S_i$, each already had 2 neighbors
in $B$ in $G$; these must be $b_1$ and $b_2$.  Hence, $G$ contains the 4-cycle
$b_1v_1b_2v_2$, contradicting the assumption that $G$ has girth at least 5.
Thus, we can apply Lemma~\ref{prop:40}, as desired; the guaranteed vertex $u$
is our desired vertex $v$, which proves the first statement of the lemma.

Now consider the second statement.
Since $d_{G}(v)\ge \sqrt{k}$, also $d_{G'}(v)\ge \sqrt{k}$.  Since
$d_{G''}(v)\le 40$, by Pigeonhole some edge $uv$ in $G'$ has multiplicity at
least $\frac{\sqrt{k}}{40}$; furthermore, these copies of the edge $uv$ are
embedded in $G'$ to create at least $\frac{\sqrt{k}}{40}-1$ consecutive
2-faces.  (It is possible that multiple copies
of $uv$, say $t$ copies, are embedded in $G'$, and thus in $G''$, such that
they do not create a 2-face.  However, now the $t$ copies of $uv$ contribute $t$
to $d_{G''}(v)$, so they do not impede this Pigeonhole argument.)
\end{proof}



Now we prove a slight strengthening of Lemma~\ref{prop:40}.

\begin{lem}
\label{prop:40:stronger}
Let $H$ be a multigraph embedded in the plane so that no face is a 2-face,
except for possibly the outer face.  
Let $w$ be some specified vertex on the outer face.
Let $A, C, D$ be disjoint vertex sets such that $w\in D$,
$V(H)=A\cup C\cup D$, every $v\in C$ satisfies $\card{N(v)\cap D}\ge 2$,
and for all $v_1,v_2\in C$ such that $\card{N(v_1)\cap D}=\card{N(v_2)\cap
D}=2$ we have $N(v_1)\cap D\ne N(v_2)\cap D$.  If $A$ is an independent set and
$d(v)\ge 3$ for all $v\in A$, then there exists $u\in D\setminus\{w\}$ such
that $d_{H[D]}(u)\le 10$ and $d_H(u)\le 40$.
\end{lem}

Before proving the lemma, we note that it implies Lemma~\ref{prop:40}, as follows.
Suppose that $H=(V,E,F)$ and $V(H)=A\cup C\cup D$, satisfy the hypothesis of
Lemma~\ref{prop:40}.  Choose an arbitrary vertex $w$ on the outer face; if it
is not in $D$, then move it to $D$.  Now the vertex $u$ guaranteed by
Lemma~\ref{prop:40:stronger} also satisfies the conclusion of Lemma~\ref{prop:40}.

\begin{proof}[Proof of Lemma~\ref{prop:40:stronger}.]
Suppose, to the contrary, that the lemma is false.  
Among all counterexamples, choose one, call it $H=(V,E,F)$ with
$V=A\cup C\cup D$ such that the following properties hold:
\begin{enumerate}
\item[(1)] $\card{D}$ is minimized; and, subject to that,
\item[(2)] the number of cut-vertices in $C\cup D$ is minimized; and, subject to
that,
\item[(3)] if possible, the outer face is a 2-face with a vertex of $A$ on its
boundary; and, subject to that,
\item[(4)] the number of parallel edges incident with vertices in $A$ is
minimized; and, subject to that,
\item[(5)] the number of edges incident to $D$ is maximized; and, subject to that,
\item[(6)] the number of edges incident to $A\cup C$ is maximized.
\end{enumerate}

We prove the lemma via a series of claims.

\begin{claim}
\label{clm0}
$H$ is connected, and no vertex of $C\cup D$ is a cut-vertex.
\end{claim}
\begin{proof}
If $H$ is disconnected, then each of its components contains a vertex of $D$, so
the component containing $w$ contradicts (1).
Thus, $H$ is connected.  Now suppose, to the contrary, that there exists $v\in
C\cup D$ that is a cut-vertex.  As a subclaim, we show that there is some
component $H_i$ of $H-v$, with $w\notin V(H_i)$, such that
$H[V(H_i)\cup\{v\}]$ has an embedding in which no face is a 2-face except for
possibly the outer face.  To see this, we consider two cases: (i) every
component of $H-v$ has a vertex on the outer face of $H$ and (ii) some component
does not.  In (i), we simply take some $H_i$ such that $w\notin V(H_i)$.  Now
$H[V(H_i)\cup\{v\}]$ has no 2-faces, except for possibly the outer face.  So we
are done.  In (ii), we take some component $H_i$ such that no other component
$H_j$ lies inside a face of $H[V(H_i)\cup\{v\}]$ in $H$.  That such a component
exists follows from the fact that every rooted tree has at least one leaf that
is not the root.  (We construct a rooted forest where every component of $H-v$
with a vertex on the outer face of $H$ is a root of its own tree, and a
component $H_j$ of $H-v$ is a child of a component $H_k$ if $H_j$ lies
inside a face of $H[V(H_k)\cup\{v\}]$.) This proves the subclaim.

Let $H'=H[V(H_i)\cup\{v\}]$.  Let $D'=(D\cap V(H'))\cup\{v\}$, $C'=(C\cap
V(H'))\setminus \{v\}$, $A'=A\cap V(H')$, and $w'=v$.  Note that $\card{D'}\le
\card{D}$, since $D'\setminus \{v\}\subseteq D$ and $w\in D\setminus D'$.
(Of course, if $v=w$, then $w\notin D\setminus D'$.  However, then we do not
move $v$ from $C'$ to $D'$, so again $|D'|\le |D|$.) Further, $v$ is a
cut-vertex of $H$, but not of $H'$.  So,
by (1) and (2) in our choice of $H$, the lemma holds for $H'$, with $w'=v$;
that is, there exists $u\in D'\setminus\{v\}$ with the desired properties.
This proves the lemma for $H$, since
$d_{H[D]}(u)=d_{H'[D]}(u)$ and $d_{H}(u)=d_{H'}(u)$.  
%
%
\end{proof}

\begin{claim}
Every $3^+$-face $f$ contains a vertex of $A$ on its boundary.
\label{clm1}
\end{claim}
\begin{proof}
Suppose, to the contrary, there exists a $3^+$-face $f$ with no vertex of $A$
on its boundary.  Now we can add a new vertex, $v$, to $A$ and make $v$ adjacent
to every vertex on $f$.  This contradicts (5) or (6) in our choice of $H$.
\end{proof}
\vspace{-.1in}

\begin{claim}
Every $3^+$-face $f$ is a 3-face, and the boundary of every $3$-face 
contains exactly one vertex of $A$.
For every vertex $u\in V(H)$ and every pair of vertices $v_1,v_2$ such that
$uv_1$ and $uv_2$ appear
consecutively in the cyclic order of edges incident to $u$, there exists an edge
$v_1v_2$ such that edges $v_1u$, $uv_2$, $v_2v_1$ induce a face of length $3$
(possibly with additional parallel edges). 
\label{clm2}
\end{claim}
\begin{proof}
Since $v_1u$ and $uv_2$ are consecutive around $u$, they lie on a $3^+$-face
$f$.  Suppose, to the contrary, that some face $f$ has length at least 4.  
By Claim~\ref{clm1}, $f$ contains a vertex $v\in A$ on its boundary.  Since
$d(v)\ge 3$, set $A$ is independent, and $D\cup C$ contains no cut-vertex, the
vertices that immediately precede and succeed $v$ on the boundary of $f$ are
distinct and are not in $A$; call them $v_1$ and $v_2$. 
Since $f$ has length at least 4, we can add the edge $v_1v_2$ in the interior of
$f$, while maintaining planarity (and without creating a 2-face).  This
contradicts (5) or (6) in our choice of $H$.  Thus, $f$ has length 3.  By
Claim~\ref{clm1}, $V(f)$ contains at least one vertex in $A$.  Since $A$ is an
independent set, $V(f)$ contains exactly one vertex of $A$.  Finally,
by definition $v_1,u,v_2$ appear along the boundary of some face, $f$.  That $f$
is a 3-face follows from the first statement.
\end{proof}
\vspace{-.1in}

\begin{claim}
The outer face is a 2-face with a vertex of $A$ on its boundary.
\label{clm3}
\end{claim}
\begin{proof}
Suppose, to the contrary, this is false.  Let $f$ denote the outer face.
By Claim~\ref{clm2}, either $f$ is a 2-face with no vertex of $A$ on its
boundary, or
$f$ is a 3-face with a vertex of $A$ on its boundary.  Suppose the former, and
let $w,v$ be the boundary vertices of $f$ (where $w$ is specified in the
hypotheses of the lemma).  Add a new vertex $a\in A$, add edge $av$
and two copies of edge $aw$, so that the outer 2-face is bounded by $a$ and $w$.
This contradicts (3) in our choice of $H$.  Instead, assume the latter, and let
$a,v,w$ be the boundary vertices of $f$.  Now add a second copy of $aw$, so that
the new outer face is a 2-face, bounded by $a$ and $w$.  Again, this contradicts
(3) in our choice of $H$.  This proves the claim.
\end{proof}

\begin{claim}
Every vertex in $C$ has at most 3 neighbors in $D$. 
\label{clm4}
\end{claim}
\begin{proof}
Suppose, to the contrary, that some vertex $u\in C$ has neighbors
$v_1,v_2,\ldots, v_p$ ($p\ge
4$) in $D$.  We add an edge $v_1v_3$ and replace $u$ by two new vertices $u_1$
and $u_2$ in $C$, where $u_1$ is adjacent to $v_1, v_2, v_3$ and $u_2$ to
$v_3,\ldots,v_p,v_1$.  
Further, $u_1$ inherits all neighbors of $u$ in $A\cup C$ that appear between
$v_1$ and $v_2$ or between $v_2$ and $v_3$ in the cyclic order of neighbors of
$u$.  Likewise, $u_2$ inherits all other neighbors of $u$ in $A\cup C$.
This contradicts (4) or (5) in our choice of $H$.  
\end{proof}
\vspace{-.1in}

\begin{claim}
No vertex in $A$ has incident parallel edges, with a single exception: on
the outer 2-face, with boundary vertices $a^*\in A$ and $w\in D$, vertex $a^*$
has exactly two edges to $w$. 
Further, $3\le d(a^*)\le 4$, and every vertex in $A\setminus
\{a^*\}$ has degree $3$ and at least one neighbor in $C$.  Finally, if
$d(a^*)=4$, then $a^*$ has a neighbor in $C$.
\label{clm5}
\end{claim}

\begin{proof}
Let $A^*=A\setminus\{a^*\}$.
Assume there is $u\in A^*$ and $v\in C\cup D$ such that $H$ has parallel edges
between $u$ and $v$; choose such a pair $u,v$ and two such edges that bound a
face $f$ in $H[\{u,v\}]$ (other than the outer face), so as to minimize the
number of vertices embedded inside $f$.  By Claim~\ref{clm0}, vertex $v$ is not
a cut-vertex, so $u$ must have a neighbor inside $f$.  Let $V'$ be the set of
vertices that lie inside $f$ (including $u$ and $v$).  Let $H'=H[V']$, let
$A'=A\cap V'$, let $C'=(C\cap V')\setminus\{v\}$, let $D'=(D\cap V')\cup\{v\}$,
and let $w'=v$.  Now $H'$ contradicts (4) in our choice of $H$. 
If $a^*$ has more than 2 edges to $w$ or has parallel edges to a vertex other
than $w$, then nearly the same argument gives a contradiction, using $a^*$ in
place of $u$.  This proves the first statement.

Suppose, to the contrary, there is a vertex $u$ in $A^*$ of degree at least $4$,
and denote its consecutive neighbors in a cyclic order by $v_1,\ldots,v_p$
($p \geq 4$).  We add an edge $v_1v_3$ and replace $u$ by two new vertices
$u_1$ and $u_2$ in $A$, where $u_1$ is adjacent to $v_1, v_2, v_3$ and $u_2$ to
$v_3,\ldots,v_p,v_1$.   This contradicts (5) or (6) in our choice of $H$.
If $d(a^*)\ge 5$, then we do something similar, as follows.  Let
$v_1,\ldots,v_p$ denote the
neighbors of $a^*$ in cyclic order, with $v_1=w$.  Remove $a^*$ and add three
new vertices $a^*_1,a^*_2,a^*_3$.  Let $a^*_1$ inherit from $a^*$ edges to
$v_1,v_2,v_3$.  Let $a^*_2$ inherit from $a^*$ edges to $v_3,\ldots,v_p,v_1$.
Lastly, let $a^*_3$ have two edges to $v_1=w$ (bounding the outer face), and
also an edge to $v_3$.  This contradicts (5) or (6) in our choice of $H$.

Finally, suppose that
$u\in A^*$ and all 3 of its neighbors are in $D$. Now we can move $u$ to $C$,
and add a new vertex $u'$ to $A$ in one of the triangular faces $f$ incident to
$u$, making $u'$ adjacent to every vertex on $f$.  This contradicts (5)
in our choice of $H$.  Thus, each vertex in $A^*$ has a neighbor in $C$.  
Nearly the same modification works if $d(a^*)=4$ and all neighbors of $a^*$ are
in $D$.  However, now we must ensure that (3) holds for the modified graph,
since (3) holds for $H$.  So we put $u'$ into the outer 2-face of $H$ and join
it with $a^*$ and (by two edges) with $w$.  This creates a new outer 2-face
bounded by $u'\in A$ and $w$.  
\end{proof}

\begin{claim}
Every face contains a vertex of $D$, and $C$ is an independent set. 
\label{clm6}
\end{claim}
\begin{proof}
Assume, for a contradiction, that there are faces
containing only vertices of $A \cup C$, and take $f$ to be one adjacent to a
face $f'$ containing a vertex of $D$. By hypothesis, $w\in D$ and $w$ is on
the outer face.  So any face sharing an edge with the outer face has a vertex
of $D$ on its boundary.  Since every non-outer face in $H$
is a 3-face, this yields a $4$-cycle $(u,v_1,v_2,v_3)$ where $u \in D$,
$v_1,v_2,v_3 \in A \cup C$ and each of $(v_3,u,v_1)$ and $(v_1,v_2,v_3)$
induces a face. By Claim~2, either $v_1\in A$ or $v_3\in A$.
By symmetry, assume that $v_1 \in A$. 

Since $A$
induces an independent set, $v_2, v_3 \in C$.  Consider the other face
$f''=(v_2,v_3,v_4)$ incident to $v_2v_3$.  Note that $v_4\ne v_1$, since
$v_1\in A$, so $d(v_1)\le4$.  Also, $v_4\ne u$, since $v_4\in A$ and $u\in D$.
Let $u'$ denote the third neighbor of $v_4$.
Now we delete the three edges $v_1v_3$, $v_2v_3$, and $v_2v_4$,
and add the two edges $uv_4$ and $u'v_1$. This contradicts (5) in our choice of
$H$, unless $u'=u$, so in the modified graph $u$ is a cut-vertex.
However, in this case we proceed as in the proof of Claim~1.  Since $v_2$ and
$v_3$ each have a neighbor in $D\setminus\{u\}$, the graph $H'$ satisfies
$\card{D'}<\card{D}$, which contradicts (1) in our choice of $H$.  So, we can
assume that $u\ne u'$.  
One other possible exception is that $v_4=a^*$ and $u'=w$, and one of the edges
from $a^*$ to $w$ stops us from adding the edge $u'v_1$.  Now we simply delete
$v_3a^*$ and add $v_2w$, contradicting (5) in our choice of $H$.
Consequently, every face contains a vertex of $D$. 

If there is an edge $v_1v_2$ between two vertices in $C$, then the two incident
3-faces must each contain a vertex in $D$, respectively $u_1$ and $u_2$. Now we
delete edge $v_1v_2$ and add an edge $u_1u_2$, which contradicts (5) in our
choice of $H$.  Thus, $C$ is an independent set.
\end{proof}

\begin{claim}
Let $H'$ denote $H[C\cup D]$, the plane multigraph induced by the vertices of
$C$ and $D$, with the plane embedding inherited from $H$.  
Every face of $H'$ is a 3-face, except possibly the outer face, and
$d_{H'}(u)=\frac{d_H(u)}2$ for every $u\in (C\cup D)\setminus \{w\}$.
\label{clm7}
\end{claim}

\begin{proof}
Let $a^*$ denote the vertex of $A$ on the outer 2-face, and let
$A^*=A\setminus\{a^*\}$.
Let $f$ be an arbitrary non-outer face of $H'$.  Since every non-outer face of
$H$ contains a vertex of $A$, face $f$ does not exist in $H$; so $f$ was formed
by deleting one or more vertices of $A$.  Since $A$ is an independent set in
$H$ and every non-outer face of $H$ is a 3-face, $f$ contained exactly one
vertex of $A$.  Since $d_H(a)=3$ for all
$a\in A^*$, face $f$ has length 3, and the vertices of $A$ and $C\cup D$
alternate in the neighborhood of every vertex $u\in (C\cup D)\setminus w$ in
$H$.  This implies that $d_{H'}(u)=\frac{d_H(u)}2$.
\end{proof}

\begin{claim}
Let $H''$ denote $H[D]$, the plane multigraph induced in $H$ by the vertices of
$D$. Form $H'''$ from $H''$ by repeatedly deleting one edge in a 2-face until
the resulting graph contains no 2-faces.  For every vertex $u\in D$ that is not
on the outer 2-face of $H'$ (if the outer face of $H'$ is a 2-face), we have
$d_{H''}(u)=\frac{d_{H'}(u)}2$ and $d_{H'''}(u)\ge \frac{d_{H''}(u)}2$. Further,
there exists vertex $u\in D\setminus\{w\}$ (also not on the outer face of $H'$
if it is a 2-face) with $d_{H'''}(u)\le 5$.  Thus, $d_D(u)=d_{H''}(u) \leq
10$ and $d_H(u) \leq 8 d_{H'''}(u) \leq 40$. 
\label{clm8}
\end{claim}

\begin{proof}
By Claim~\ref{clm7}, every face of $H'$ is a 3-face, except possibly the outer
face.  Further, each such 3-face has a vertex of $C$ on its boundary, since no
vertex in $A$ has all three neighbors in $D$, by Claim~\ref{clm5}.  Since $C$ is
an independent set (by Claim~\ref{clm6}), there is in fact a bijection between
$C$ and the faces of $H''$ (excluding the outer face of $H''$
if the outer face of $H'$ is a 2-face bounded by two vertices in $D$).
So, for every vertex $u \in D$ not on an outer 2-face of $H'$, we have
$d_{H''}(u)=\frac{d_{H'}(u)}{2}$. By hypothesis on $H$, no two vertices of $C$,
each with exactly two neighbors in $D$, have the same two neighbors in $D$.
Thus, no two faces of length $2$ in $H''$ share an edge (except that one can
possibly share an edge with the outer 2-face, if it exists). 
Therefore, by deleting exactly one edge in every face of length $2$, we obtain
a plane multigraph $H'''$ with no 2-face except possibly the outer face, and
such that for every vertex $u \in D$ (not on an outer 2-face of $H'$) we have
$d_{H'''}(u)\geq \frac{d_{H''}(u)}{2}$. 
Since every non-outer face of $H'''$ is a $3^+$-face,
Euler's formula implies that $\card{E(H''')}\le 3\card{D}-5$.
Thus, there is some $u\in D$ that is not on an outer 2-face of $H'$ such that
$d_{H'''}(u) \leq 5$. As a consequence, $d_D(u)=d_{H''}(u) \leq 10$ and
$d_H(u) \leq 8 d_{H'''}(u) \leq 40$. 
\end{proof}

Claim~\ref{clm8} completes the proof of the lemma.
\end{proof}

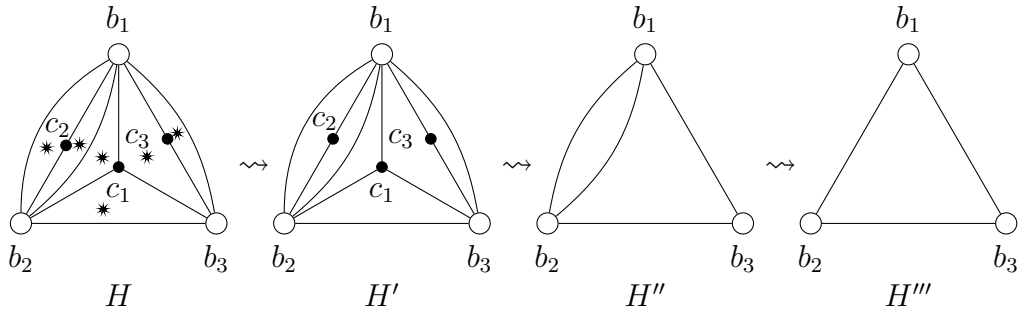
\begin{figure}[!h]
\center
\begin{tikzpicture}[scale=1.0]
    \tikzstyle{whitenode}=[draw,ellipse,fill=white,minimum size=8pt,inner sep=0pt]
    \tikzstyle{blacknode}=[draw,circle,fill=black,minimum size=4pt,inner sep=0pt]
     \tikzstyle{A}=[draw,fill=black,star,star point ratio=3.8,star points=9,minimum size=3pt,inner sep=0pt]

\draw (0,0) node[blacknode] (c1) [label=-90:$c_1$] {};
\draw (c1) ++(90:1.5cm) node[whitenode] (b1) [label=90:$b_1$] {};
\draw (c1) ++(-150:1.5cm) node[whitenode] (b2) [label=-90:$b_2$] {};
\draw (c1) ++(-30:1.5cm) node[whitenode] (b3) [label=-90:$b_3$] {};

\draw (c1) ++(-110:0.6cm) node[A] (a2) {};
\draw (c1) ++(20:0.4cm) node[A] (a3) {};
\draw (c1) ++(30:0.9cm) node[A] (a3b) {};
\draw (c1) ++(150:0.25cm) node[A] (a1) {};
\draw (c1) ++(150:0.6cm) node[A] (a1b) {};
\draw (c1) ++(165:0.99cm) node[A] (a1c) {};

\draw (b2) --++(60:1.199cm) node[blacknode] (c2) {};
\draw (c2) ++(115:0.26cm) node {$c_2$};
\draw (b3) --++(120:1.299cm) node[blacknode] (c3) [label=180:$c_3$] {};
\draw (c2) -- (b1);
\draw (c3) -- (b1);

\draw[bend left=20] (b1) edge node {} (b2);
\draw[bend left=30] (b2) edge node {} (b1);
\draw[bend left=20] (b1) edge node {} (b3);
\draw (b2) -- (b3);
\draw (c1) -- (b1);
\draw (c1) -- (b2);
\draw (c1) -- (b3);

\draw (c1) ++(-90:1.7cm) node {$H$};
\draw (c1) ++(0:1.8cm) node {$\leadsto$};

\draw (3.5,0) node[blacknode] (c1) [label=-90:$c_1$] {};
\draw (c1) ++(90:1.5cm) node[whitenode] (b1) [label=90:$b_1$] {};
\draw (c1) ++(-150:1.5cm) node[whitenode] (b2) [label=-90:$b_2$] {};
\draw (c1) ++(-30:1.5cm) node[whitenode] (b3) [label=-90:$b_3$] {};

\draw (b2) --++(60:1.299cm) node[blacknode] (c2) {};
\draw (c2) ++(115:0.26cm) node {$c_2$};
\draw (b3) --++(120:1.299cm) node[blacknode] (c3) [label=180:$c_3$] {};
\draw (c2) -- (b1);
\draw (c3) -- (b1);

\draw[bend left=20] (b1) edge node {} (b2);
\draw[bend left=30] (b2) edge node {} (b1);
\draw[bend left=20] (b1) edge node {} (b3);
\draw (b2) -- (b3);
\draw (c1) -- (b1);
\draw (c1) -- (b2);
\draw (c1) -- (b3);

\draw (c1) ++(-90:1.7cm) node {$H'$};
\draw (c1) ++(0:1.8cm) node {$\leadsto$};

\draw (7,0) node (c1) {};
\draw (c1) ++(90:1.5cm) node[whitenode] (b1) [label=90:$b_1$] {};
\draw (c1) ++(-150:1.5cm) node[whitenode] (b2) [label=-90:$b_2$] {};
\draw (c1) ++(-30:1.5cm) node[whitenode] (b3) [label=-90:$b_3$] {};

\draw[bend left=20] (b1) edge node {} (b2);
\draw[bend left=20] (b2) edge node {} (b1);
\draw[bend left=0] (b1) edge node {} (b3);
\draw (b2) -- (b3);

\draw (c1) ++(-90:1.7cm) node {$H''$};
\draw (c1) ++(0:1.8cm) node {$\leadsto$};

\draw (10.5,0) node (c1) {};
\draw (c1) ++(90:1.5cm) node[whitenode] (b1) [label=90:$b_1$] {};
\draw (c1) ++(-150:1.5cm) node[whitenode] (b2) [label=-90:$b_2$] {};
\draw (c1) ++(-30:1.5cm) node[whitenode] (b3) [label=-90:$b_3$] {};

\draw (b2) -- (b3);
\draw (b2) -- (b1);
\draw (b1) -- (b3);

\draw (c1) ++(-90:1.7cm) node {$H'''$};

\end{tikzpicture}
\caption[]{The evolution, in the proof of Lemma~\ref{prop:40:stronger}, from
$H$ to $H'''$ of a subgraph of $H$.  Here $b_1, b_2, b_3 \in D$, $c_1, c_2, c_3
\in C$. Each \!\!\!
{
\begin{tikzpicture}
\draw[white] (0,0) -- (0,.001);
\tikzstyle{A}=[draw,fill=black,star,star point ratio=3.8,star points=9,minimum size=3pt,inner sep=0pt]
\draw (0,.10) node[A] () {};
\end{tikzpicture}
}
\!\!\! represents a vertex in $A$ adjacent to all three
vertices incident to the face. (Black vertices have all incident edges drawn, but white
vertices may have more incident edges.)}
\label{fig:prop:40:stronger}
\end{figure}

Recall that $S'=V(G')\setminus B$.

\begin{lemma}
\label{lem:multipleedges}
No vertex in $S'$ is incident to $3$ or more consecutive faces of length $2$ in $G'$.
\end{lemma}
\begin{proof}
Assume, for a contradiction, that there is an edge $uv$ in $G'$,
with $u \in S'$, such that at least 3 consecutive faces have boundary $(u,v)$. 
First consider the case where $v \in S'$. In the
construction of $G'$ from $G$, an edge is added between $u$ and $v$ only when
there is a vertex of degree $2$ adjacent to both $u$ and $v$ that is suppressed.
Hence, regardless of whether $uv$ belongs to $E(G)$ or is formed from the
suppression of a vertex of degree $2$ adjacent to $u$ and $v$,
there exists a cycle in $G$ of length at most $4$, contradicting that $G$
has girth at least five.

So we may assume that $v \in B$. Let $T_1=S_1\cap N(v)$, that is the set of
small neighbors of $v$ with exactly one big neighbor (which must be $v$). In
the construction of $G'$, an edge is added between $u$ and $v$ only if either
there is a neighbor of $u$ in $T_1$, or if there is a vertex of degree $2$
adjacent to both $u$ and to a vertex in $T_1$. Let $U_1$ denote the set of
vertices of degree $2$ that are adjacent to $u$ and also to a vertex in $T_1$. 
Note that each copy of $uv$ in $G'$ corresponds to a path of length at most 3 in
$G$ with all vertices in $\{u,v\}\cup U_1\cup T_1$.  Thus, $uv\notin E(G)$,
since this would create a 3-cycle or 4-cycle in $G$, contradicting that $G$ has
girth at least 5.

Since $uv$ has multiplicity at least four, and the copies of $uv$ form 3
consecutive faces of length 2 in $G'$, we know $4\le |U_1|+|N(u) \cap T_1|$; 
further, there exist four vertices, $w_1,\ldots,w_4$, in $T_1 \cup U_1$ that
are consecutive in the cyclic neighborhood of $u$ in $G$. Since $G$ has girth
at least five, $|N(u) \cap N(v)|\leq 1$, so $|N(u) \cap T_1| \leq 1$.  Thus, at
least one of $w_2$ and $w_3$ is in $U_1$; by symmetry, assume $w_2\in U_1$. 
Let $x_2$ be the neighbor of $w_2$ in $G$ distinct from $u$. Note that $x_2 \in
T_1$. Since $G$ has girth at least five, $x_2$ is neither adjacent to a
neighbor of $v$ nor to a neighbor of a neighbor of $v$. Regardless of whether
$w_3$ belongs to $U_1$ or $T_1$, it follows by planarity that $w_2$ and $v$ are
the only neighbors of $x_2$ in $G$. Therefore $d(x_2)=2$, which is a
contradiction to Lemma~\ref{lem:degtwo}, since $w_2 \not\in N(B)$.
\end{proof}

Now we use Lemma~\ref{lem:multipleedges} to strengthen the final conclusion of
Lemma~\ref{lem:smallbigdeg}.
Recall that $G''$ is formed from $G'$ by suppressing 2-faces.

\begin{corollary}
\label{cor:highmultipB}
There exist $u,v \in B \cap V(G')$ such that there are at least 
$\frac{\sqrt{k}}{10}-13$ consecutive faces of length two with boundary $(u,v)$
in $G'$.
\end{corollary}
\begin{proof}
Let $v$ be as in Lemma~\ref{lem:smallbigdeg}.  By
Lemma~\ref{lem:multipleedges}, each edge from $v$ to a small vertex in $G''$
accounts for at most 3 consecutive edges in $G'$ incident to $v$.  Thus, the
$|N_{G''}(v)\cap B|$ big neighbors of $v$ account for the remaining at least
$d_G(v)-3(40)$ edges.  Since $v$ is big, $d_{G'}(v)\ge d_G(v)\ge \sqrt{k}$. 
By Pigeonhole, some edge from $v$ to
a big vertex, say $u$, accounts for at least $(\sqrt{k}-120)/d_{G''[B]}(v)$
consecutive edges incident to $v$ in $G'$; and the number of consecutive
faces with boundary $(u,v)$ is one less.  Since $d_{G''[B]}(v)\le 10$,
some big neighbor $u$ shares with $v$ at least $\frac{\sqrt{k}}{10}-13$
consecutive faces of length 2.
\end{proof}



If $b_1,b_2 \in B\cap V(G')$ are such that at least $r$ consecutive faces
$f'_1,\ldots,f'_r$ of $G'$ have boundary $(b_1,b_2)$, then these faces 
are an \emph{$r$-region} $R'$ of $G'$; see Figure~\ref{fig:regions}.
Analogously, an \emph{$r$-region} $R$ of $G$ is a set of faces which contract
to an $r$-region $R'$ in $G'$.
We define $V(R)$ as $(\bigcup_{i=1}^rV(f_i))\setminus\{b_1,b_2\}$.


\begin{figure}[!h]
\center
\begin{tikzpicture}[scale=1.2]
    \tikzstyle{whitenode}=[draw,ellipse,fill=white,minimum size=8pt,inner sep=0pt]
    \tikzstyle{blacknode}=[draw,circle,fill=black,minimum size=4pt,inner sep=0pt]
\draw (0,-0.5) node[whitenode] (b1) [label=0:$b_1$] {\small{$\sqrt{k}^+$}} 
 ++(-90:3.5cm) node[whitenode] (b2) [label=-0:$b_2$] {\small{$\sqrt{k}^+$}}
 ++(-90:1cm) node {$G$};

\draw (-1,-1.5) node[whitenode] (b11) {};
\draw (-0.3,-1.5) node[blacknode] (b12) {};
\draw (0.4,-1.5) node[blacknode] (b13) {};
\draw (1,-1.5) node[blacknode] (b14) {};
\draw (1.5,-1.5) node[whitenode] (b15) {};

\draw (-1,-3) node[whitenode] (b21) {};
\draw (-0.2,-3) node[blacknode] (b22) {};
\draw (0.6,-3) node[blacknode] (b23) {};
\draw (1.3,-3) node[whitenode] (b24) {};

\draw (b11) --++ (-90:0.75cm) node[blacknode] (d1) {};
\draw (b12) --++ (-90:0.75cm) node[blacknode] (d2) {};

\draw (d1) edge  node {} (b21);
\draw (d2) edge  node {} (b22);
\draw (b1) edge node {} (b15);

\foreach \i in {1,...,4}
{
\draw (b1) edge node {} (b1\i);
\draw (b2) edge node {} (b2\i);
\pgfmathtruncatemacro{\j}{\i+1}; 
\draw (b2\i) edge node {} (b1\j);
}

\draw (-0.65,-1.7) node {$f_1$};
\draw (-0.6,-3) node {$f_2$};
\draw (-0.05,-1.8) node {$f_3$};
\draw (0.35,-2.25) node {$f_4$};
\draw (1.05,-2.25) node {$f_5$};

\draw[thick,densely dotted] (-1.3,-1.2) rectangle node {} (1.8,-1.8);
\draw (2.24,-1.5) node {$B_1$};

\draw[thick,densely dotted] (-1.2,-2.05) rectangle node {} (-0.1,-2.45);
\draw (-1.6,-2.25) node {$D_2$};

\draw[thick,densely dotted] (-1.3,-2.7) rectangle node {} (1.6,-3.3);
\draw (2.04,-3) node {$B_2$};

\draw (3,-2.25) node {$\leadsto$};

\draw (5,-0.5) node[whitenode] (b1) [label=0:$b_1$] {\small{$\sqrt{k}^+$}} 
 ++(-90:3.5cm) node[whitenode] (b2) [label=-0:$b_2$] {\small{$\sqrt{k}^+$}}
  ++(-90:1cm) node {$G'$};

\draw[bend left=70] (b1) edge node {} (b2);
\draw[bend left=40] (b1) edge node {} (b2);
\draw[bend left=15] (b1) edge node {} (b2);
\draw[bend left=-15] (b1) edge node {} (b2);
\draw[bend left=-40] (b1) edge node {} (b2);
\draw[bend left=-70] (b1) edge node {} (b2);

\draw (4,-2.25) node {$f_1$};
\draw (4.5,-2.25) node {$f_2$};
\draw (5,-2.25) node {$f_3$};
\draw (5.5,-2.25) node {$f_4$};
\draw (6,-2.25) node {$f_5$};


\end{tikzpicture}
\caption{A $5$-region in $G$ and the corresponding $5$-region in $G'$.}
\label{fig:regions}
\end{figure}
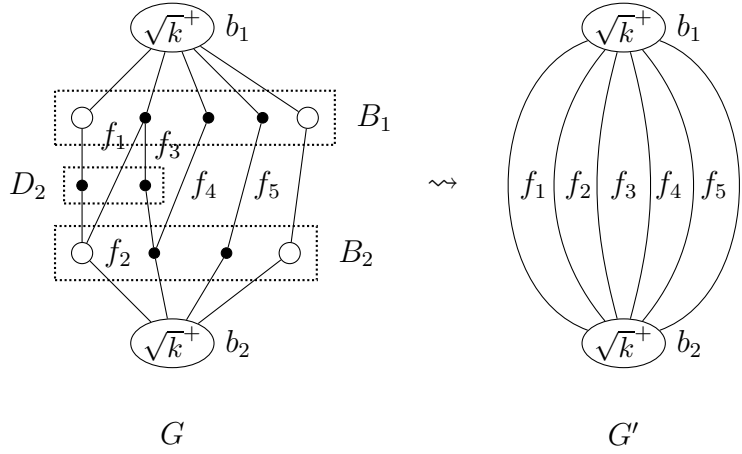

\begin{observation}
If $R$ is an $r$-region of $G$, then $V(R)=B_1\cup B_2 \cup D_2$, where
$B_1$, $B_2$, and $D_2$ are disjoint vertex sets such that
$B_1\subseteq N(b_1)$ and $B_2\subseteq N(b_2)$ for some $b_1,b_2\in B$;
further, $D_2$ is an independent set of degree two vertices, each of which has
one neighbor in $B_1$ and the other neighbor in $B_2$.
\label{rregionObs}
\end{observation}
\begin{proof}
Let $R$ be an $r$-region of $G$.  By definition, there exist $b_1,b_2\in B\cap
V(G')$ such that the $r$-region $R'$ consists of at least $r$
consecutive faces in $G'$, each with boundary $(b_1,b_2)$.
Recall that $G'$ is formed from $G$ by suppressing the degree 2 vertices in $S\setminus
N(B)$ and contracting each edge joining $S_1$ and $B$.  By Lemma~\ref{lem:degtwo}, these
suppressed degree 2 vertices form an independent set.  Thus, each copy of
$b_1b_2$ in $R'$ in $G'$ corresponds to a path of length 1, 3, or 4 joining
$b_1$ and $b_2$ in $G$.  Each such path $P$ (of length at least 3) must contain
a vertex from each of $B_1$ and $B_2$.  If $P$ contains another vertex $w$,
then $w$ must be suppressed in forming $G'$, so $w$ must be a degree 2 vertex
with a neighbor in each of $B_1$ and $B_2$.  This proves the observation.
\end{proof}

Hereafter, we use $B_1$, $B_2$, $D_2$, $b_1$, $b_2$, and $V(R)$ as defined in
the previous observation.

\begin{lemma}
\label{lem:smalldegreeinregions}
If $R$ is an $r$-region of $G$, then $B_1$ and $B_2$ are independent sets,
and each $v\in B_1\cup B_2$ satisfies $|N(v)\cap V(R)|\le 3$.
\end{lemma}
\begin{proof}
The fact that $B_1$ and $B_2$ are independent sets follows from
the assumption that $G$ has girth at least five. 
Now choose $v\in B_1\cup B_2$ and suppose, for a contradiction, that $|N(v)\cap
V(R)|\ge 4$.  Without loss of generality, we may assume that $v\in B_1$; 
see Figure~\ref{fig:smalldegreeinregions}.
Recall that $V(R)\subseteq B_1\cup B_2\cup D_2$.  Since $B_1$ is independent,
$|N(v)\cap B_1|=0$; thus $|N(v)\cap (B_2\cup D_2)|\ge 4$. Since $G$ has girth at least
five, $|N(v)\cap B_2|\le 1$, so $|N(v)\cap D_2|\ge 3$. Hence, by planarity,
there exists $u\in N(v) \cap D_2$ such that if $w$ is the other neighbor of
$u$, then $vw \in E(G')$ (actually $v$ gets contracted into $b_1$ and $w$ gets
contracted into $b_2$ when forming $G'$) and $vw$ is incident with two faces,
each of length two, in region $R'$. Since $G$ has girth at least five, it
follows that $w$ has degree two in $G$. But now $u$ and $w$ are adjacent vertices of
degree two, yet $u\not\in N(B)$, which contradicts Lemma~\ref{lem:degtwo}.
\end{proof}

\begin{figure}[!h]
\center
\begin{tikzpicture}
    \tikzstyle{whitenode}=[draw,ellipse,fill=white,minimum size=8pt,inner sep=0pt]
    \tikzstyle{blacknode}=[draw,circle,fill=black,minimum size=4pt,inner sep=0pt]
\draw (0,0) node[whitenode] (b1) [label=0:$b_1$] {\small{$\sqrt{k}^+$}} 
-- ++(-90:1cm) node[whitenode] (v) [label=-0:$v$] {}
-- ++(-90:1cm) node[blacknode] (u) [label=-0:$u$] {}
-- ++(-90:1cm) node[blacknode] (w) [label=-0:$w$] {}
-- ++(-90:1cm) node[whitenode] (b2) [label=-0:$b_2$] {\small{$\sqrt{k}^+$}}
 ++(-90:1cm) node {$G$};

\draw (-1,-2) node[blacknode] (v1) {};
\draw (1,-2) node[blacknode] (v2) {};

\draw (-1,-3) node[whitenode] (w1) {};
\draw (1,-3) node[whitenode] (w2) {};

\foreach \i in {1,2}
{
\draw (v\i) -- (v);
\draw (v\i) -- (w\i);
\draw (w\i) -- (b2);
}

\draw[thick,densely dotted] (-1.3,-1.7) rectangle node {} (1.3,-2.3);
\draw (2.04,-2) node {$D_2$};

\draw[thick,densely dotted] (-1.3,-2.7) rectangle node {} (1.3,-3.3);
\draw (2.04,-3) node {$B_2$};

\draw (3,-2.5) node {$\leadsto$};

\draw (5,0) node[whitenode] (b1) [label=0:$b_1$] {\small{$\sqrt{k}^+$}} 
 ++(-90:4cm) node[whitenode] (b2) [label=-0:$b_2$] {\small{$\sqrt{k}^+$}}
  ++(-90:1cm) node {$G'$};
  
  \draw[bend left=30] (b1) edge node {} (b2);
  \draw[bend left=-30] (b1) edge node {} (b2);
  \draw[bend left=0] (b1) edge node {} (b2);


\end{tikzpicture}
\caption{An illustration of the proof of Lemma~\ref{lem:smalldegreeinregions}.}
\label{fig:smalldegreeinregions}
\end{figure}
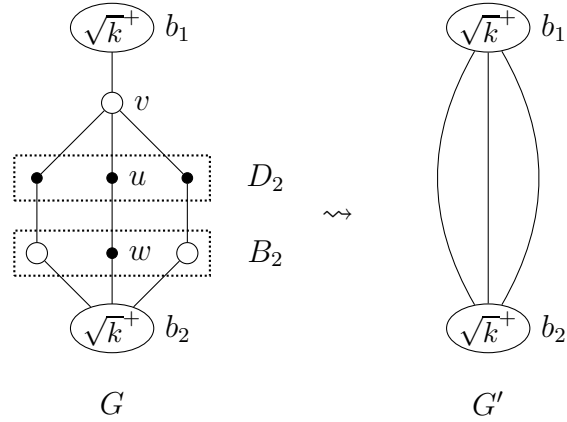

To complete the proof of Theorem~\ref{th:main}, we need one more reducible
configuration; in Lemma~\ref{lem:smallregion}, we show that an $r$-region is
reducible, if $r\ge 161$.  Before that, we need two lemmas about list-coloring.
 The first played a key role in Galvin's proof~\cite{Galvin95} that
$\chil'(G)=\Delta(G)$ for every bipartite graph $G$ (here $\chil'$ denotes the
edge list chromatic number).

A \emph{kernel} in a digraph $D$ is an independent set $F$ of vertices such that
each vertex in $V(D)\setminus F$ has an out-neighbor in $F$. A digraph $D$ is
\emph{kernel-perfect} if for every $A\subseteq V(D)$, the digraph $D[A]$ has a
kernel. 
To prove our next result,  
we will need the following lemma of Bondy, Boppana, and Siegel
(see~\cite[p.~129]{AlonT92} and \cite[p.~155]{Galvin95}).
For completeness, we include an easy proof.

\begin{lemma}\label{lem:kernelperfect}
Let $D$ be a kernel-perfect digraph with underlying graph $G$. If $L$ is a
list-assignment of $V(G)$ such that for all $v\in V(G)$,
$$|L(v)|\ge d^{+}(v) + 1,$$
then $G$ is $L$-colorable.
\end{lemma}
\begin{proof}
We use induction on $|V(G)|$.  Choose some color $c\in \cup_{v\in V(G)}L(v)$.
Let $A_c$ be the set of vertices with color $c$ in their lists.  By assumption,
$D[A_c]$ contains a kernel, $F_c$.  Use color $c$ on each vertex of $F_c$.
Now let $D' = D\setminus F_c$ and $L'(v)=L(v)-c$ for each $v\in V(D')$.  By
induction, the remaining uncolored digraph $D'$ can be colored from its lists
$L'$; we must only check that $D'$ and $L'$ satisfy the hypothesis of the
lemma.  Since $D$ is kernel-perfect, so is $D'$.  Further, each
vertex of $D'$ lost at most one color from its list (namely, $c$). More
precisely, each vertex of $A_c\setminus F_c$ lost one color from its list and
each other vertex lost no colors.
Fortunately, since $F_c$ is a kernel for $A_c$, we get $d^+_{D'}(v)\le
d^+_D(v)-1$ for each $v\in A_c\setminus F_c$.  Thus, $|L'(v)|\ge
d^+_{D'}(v)+1$ for every $v\in V(D')$, as desired.
\end{proof}

We now use Lemma~\ref{lem:kernelperfect} to prove the following lemma, which we will 
use to show that large regions are reducible for square $(\Delta+2)$-choosability.

\begin{lemma}
\label{lem:twocliques}
Let $H$ be a graph covered by two disjoint cliques $B_1$ and $B_2$, $L$ be a
list-assignment for $V(H)$, and $S_1\subseteq B_1$ and $S_2 \subseteq B_2$ be
such that

\begin{itemize}
\item if $v\in B_i$, then $|N(v)\cap V(B_{3-i})|\le 3$,
\item if $v\in B_i\setminus S_i$, then $|L(v)|\ge |B_i|$,
\item if $v\in S_i$, then $|L(v)|\ge |B_i|-3$.
\end{itemize}

Now if $|B_1|\ge 46$, $|B_2|\ge 46$, $|S_1|\le 12$, and $|S_2|\le 12$, then
$H$ is $L$-colorable.  
\end{lemma}
\begin{proof}
We construct a kernel-perfect orientation $D$ of $H$ satisfying
Lemma~\ref{lem:kernelperfect} as follows. Let $x_1,x_2,\ldots, x_{|B_1|}$ be an
ordering of the vertices of $B_1$ and $y_1,y_2,\ldots, y_{|B_2|}$ be an
ordering of the vertices of $B_2$ such that 

\begin{itemize}
\item $x_i \in S_1$ iff $1\le i \le |S_1|$, and
\item $y_i \in S_2$ iff $1\le i \le |S_2|$, and
\item $N_{B_2}(x_a) \cap N_{B_2}(N_{B_1}(y_b))  = \emptyset$ for each $a$ and
$b$ such that $|B_1|-2\le a\le |B_1|$ and $|B_2|-2\le b\le |B_2|$.
\end{itemize}

It is helpful to restate the third condition in words: there is no path of
length 1 or 3 that starts at one of the final 3 vertices in $B_1$, ends at one
of the final 3 vertices in $B_2$, and alternates between $B_1$ and $B_2$.
We claim that such an ordering exists. To see this, let the vertices of $S_1$
be $x_1, \ldots, x_{|S_1|}$ in any order and similarly for $S_2$. Now it
suffices to ensure the third condition holds.  Note that $|B_1|-3|S_2|\ge 10$.
Suppose there exists $u\in B_1\setminus N(S_2)$ with $d_{B_2}(u)=3$.  Choose
$N_{B_2}(u)$ to be the three final vertices of $B_2$, and call this set $Z$.

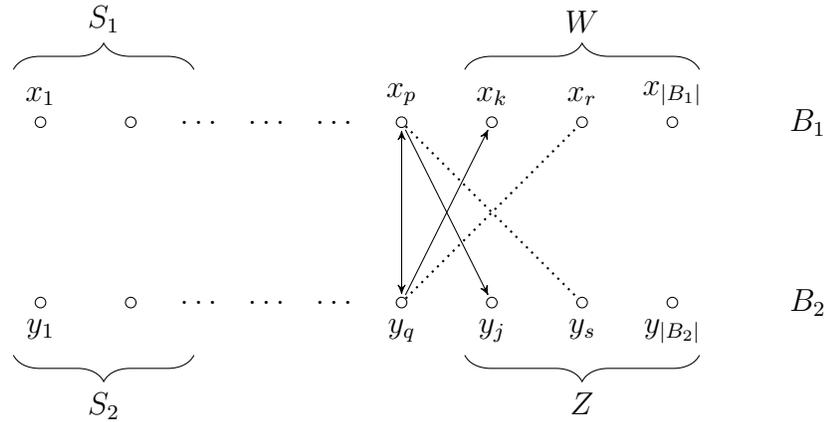
\begin{figure}[!h]
\center
\begin{tikzpicture}[scale=1.2]
    \tikzstyle{whitenode}=[draw,ellipse,fill=white,minimum size=4pt,inner sep=0pt]
    \tikzstyle{blacknode}=[draw,circle,fill=black,minimum size=4pt,inner sep=0pt]

\draw (0,0) node[whitenode] (x1) [label=90:$x_1$] {};
\draw (1,0) node[whitenode] (x2) {};
\draw (1.75,0) node {$\ldots$};
\draw (2.5,0) node {$\ldots$};
\draw (3.25,0) node {$\ldots$};
\draw (4,0) node[whitenode] (xp) [label=90:$x_p$] {};
\draw (5,0) node[whitenode] (xk) [label=90:$x_k$] {};
\draw (6,0) node[whitenode] (xr) [label=90:$x_r$] {};
\draw (7,0) node[whitenode] (xB) [label=90:$x_{|B_1|}$] {};
\draw (8.5,0) node {$B_1$};

\draw (0,-2) node[whitenode] (y1) [label=-90:$y_1$] {};
\draw (1,-2) node[whitenode] (y2) {};
\draw (1.75,-2) node {$\ldots$};
\draw (2.5,-2) node {$\ldots$};
\draw (3.25,-2) node {$\ldots$};
\draw (4,-2) node[whitenode] (yq) [label=-90:$y_q$] {};
\draw (5,-2) node[whitenode] (yj) [label=-90:$y_j$] {};
\draw (6,-2) node[whitenode] (ys) [label=-90:$y_s$] {};
\draw (7,-2) node[whitenode] (yB) [label=-90:$y_{|B_2|}$] {};
\draw (8.5,-2) node {$B_2$};

\draw [decoration={brace,amplitude=10pt,raise=0.7cm},decorate] (-0.3,0) -- (1.7,0) node [pos=0.5,anchor=south,yshift=1.05cm] {$S_1$};
\draw [decoration={brace,amplitude=10pt,mirror,raise=0.7cm},decorate] (-0.3,-2) -- (1.7,-2) node [pos=0.5,anchor=north,yshift=-1.05cm] {$S_2$};

\draw [decoration={brace,amplitude=10pt,raise=0.7cm},decorate] (4.7,0) -- (7.3,0) node [pos=0.5,anchor=south,yshift=1.05cm] {$W$};
\draw [decoration={brace,amplitude=10pt,mirror,raise=0.7cm},decorate] (4.7,-2) -- (7.3,-2) node [pos=0.5,anchor=north,yshift=-1.05cm] {$Z$};

\draw[<->,>=stealth',shorten <=1pt,shorten >=1pt] (xp) edge node {} (yq);
\draw[->,>=stealth',shorten <=1pt,shorten >=1pt] (xp) edge node {} (yj);
\draw[->,>=stealth',shorten <=1pt,shorten >=1pt] (yq) edge node {} (xk);
\draw[dotted, thick] (xp) edge node {} (ys);
\draw[dotted, thick] (yq) edge node {} (xr);

\end{tikzpicture}
\caption{The proof of Lemma~\ref{lem:twocliques}, constructing the orientation
$D$ of $H$, which shows that this situation cannot occur, due to our choices of
$W$ and $Z$.}
\label{fig:twocliques}
\end{figure}

Now $|N_{B_1}(Z)\setminus\{u\}|\le 6$, so $|N_{B_2}(N_{B_1}(Z))\setminus Z|\le
6(2)=12$ and $|N_{B_1}(N_{B_2}(N_{B_1}(Z))\setminus Z)|\le 12(2)+6=30$.
Since $|B_1|-|S_1|-30-|\{u\}|\ge 3$, we can choose the desired 3 final
vertices of $B_1$; call this set $W$.  If no such $u$ exists, then there exist
3 vertices $v_1,v_2,v_3\in B_1\setminus (S_1\cup N(S_2))$ such that
$d_{B_2}(v_i)\le 2$ for all $i\in\{1,2,3\}$.  Now swap the roles of $B_1$ and
$B_2$ and let $Z=\{v_1,v_2,v_3\}$.  The analysis is essentially the same,
except that now we have no vertex $u$.  This proves the claim that such an
ordering exists.

Let $D$ be obtained from $H$ by directing the edges of $H$ as follows. 
For each edge with both endpoints in $B_1$ or both endpoints in $B_2$, direct
the edge from the vertex with higher index to the vertex with lower index.
For each edge between $B_1$ and $B_2$, direct the edge in both directions,
unless one endpoint is among the final three vertices of $B_1$ or $B_2$; in that
case, only direct the edge into the vertex among the final three (recall that no
edge has one endpoint among the final three vertices of $B_1$ and the other
endpoint among the final three vertices of $B_2$).

We claim that $D$ is a kernel-perfect orientation. Let $A\subseteq V(H)$. Let
$p=\min \{i: x_i \in A\}$ and $q=\min \{j: y_j\in A\}$. If $A\cap
V(B_2)=\emptyset$, then $\{x_p\}$ is a kernel of $A$ as desired. Similarly if
$A\cap V(B_1)=\emptyset$, then $\{y_q\}$ is a kernel of $A$ as desired. So we may
assume that $A\cap V(B_1)\ne \emptyset$ and $A\cap V(B_2)\ne \emptyset$. If
$x_py_q\not\in E(H)$, then $\{x_p,y_q\}$ is a kernel of $A$ as desired. 
So we assume that $x_py_q\in E(H)$. 

Let $r=\min \{k: x_k \in A, x_k \not\in N_H(y_q) \}$ and $s=\min \{\ell: y_\ell
\in A, y_\ell \not\in N_H(x_p) \}$.  Now $\{x_p, y_s\}$ is a kernel of $A$,
unless there exists $j$ with $q\le j < s$ such that $y_j\in A$ and $x_py_j$ is
either not an edge of $H$ or is only directed from $x_p$ to $y_j$. 
Given the choice of $s$, it must be that $x_py_j$ is
only directed from $x_p$ to $y_j$.  (If $\{\ell:y_{\ell}\in A,y_{\ell}\notin
N(x_p)\}=\emptyset$, then the same argument works with $\{x_p\}$ in place of
$\{x_p,y_s\}$.)
Thus, we conclude that $y_j$ is among the final
3 vertices of $B_2$.  Now, we instead take as our kernel $\{y_q,x_r\}$.
This is a kernel unless there exists $k$ with $p\le k<r$ such that $x_k\in A$
and either $x_ky_q$ is not an edge or it is only directed from $y_q$ to $x_k$. 
Given our choice of $r$, we know that $x_ky_q$ is an edge.  But if $x_ky_q$ is
only directed from $y_q$ to $x_k$, then $x_k$ is among the final 3 vertices of
$B_1$.  (Similar to above, if $\{k:x_k\in A,x_k\notin N(y_q)\}=\emptyset$, then
we use $\{y_q\}$ in place of $\{y_q,x_r\}$.) However, this is impossible, since
now the path $x_ky_qx_py_j$ contradicts the third condition.  Thus, $D$ is
kernel-perfect, as desired.

Finally, we claim that $|L(v)|\ge d^+_D(v)+1$ for all $v\in V(H)$.  First
suppose that $v\in S_1\cup S_2$.  Now $v$ has at most 11 out-neighbors
within its clique and at most 3 out-neighbors in the other clique, so
$d^+_D(v)\le 14$.  Since $|B_1|\ge 18$ and $|B_2|\ge 18$, we have $|L(v)|\ge
|B_i|-3\ge 15\ge d^+_D(v)+1$.  Next, suppose that $v\in (B_1\cup B_2)\setminus
(S_1\cup S_2)$, but $v$ is not among the final 3 vertices of either $B_i$.  
By symmetry, we can assume that $v\in B_1$.  Since $v$ has no
out-neighbors among the final 3 vertices of $B_1$, it has at most $|B_1|-4$
out-neighbors in $B_1$.  Since $v$ has at most 3
out-neighbors in $B_2$, we have $|L(v)|\ge |B_1|=(|B_1|-4)+3+1\ge d^+_D(v)+1$.
Now suppose that $v$ is among the final 3 vertices of $B_1$ or $B_2$; by
symmetry, assume that $v\in B_1$.  Since all out-neighbors of $v$ are in $B_1$,
we get $d^+_D(v)\le |B_1|-1$; thus, $|L(v)|\ge d^+_D(v)+1$.
\end{proof}


\begin{lemma}
\label{lem:smallregion}
For every $r\ge 161$, graph $G$ does not have an $r$-region.
\end{lemma}
\begin{proof}
Suppose, to the contrary, that $G$ has such an $r$-region $R$, with $r\ge 161$.
Let $B_1, B_2, D_2, b_1$, and $b_2$ be as in Observation~\ref{rregionObs}.
Let $v_1$ and $v_2$ be adjacent vertices of $B_1\cup B_2\cup D_2$ such that
every vertex within distance 2 in $G$ of $v_1$ or $v_2$ is 
in $\{b_1,b_2\}\cup N(b_1)\cup N(b_2)\cup V(R)$.
To see that such vertices exist, pick $v_1\in B_1$ such that each face
containing $v_1$ is in $R$, and let $v_2$ be a neighbor of $v_1$ in $B_2\cup
D$.  By the minimality of $G$, we can $L$-color $(G-v_1v_2)^2$; call this
coloring $\varphi$.  Now we uncolor many of the vertices in $V(R)$ and extend
the coloring to $G$ using Lemma~\ref{lem:twocliques}, as well as greedily
coloring vertices of $D_2$ last.  The details forthwith.

Note that $\card{N(b_1)\cap N(b_2)}\le 1$, since $G$ has girth at least 5.
In fact, if $v\in N(b_1)\cap N(b_2)$, then $v\in V(G')$, since $v$ is neither
suppressed nor contracted into a big vertex.  Thus, $v\notin V(R)$.
So $V(R)\cap N(b_1)\cap N(b_2)=\emptyset$. 

Let $S$ be the set of vertices in $B_1\cup B_2$ that are
incident with a face of $G$ not in $R$. 
Let $B_1'=B_1\setminus N[S]$ and $B_2'=B_2\setminus
N[S]$. Note that $B_1'$ and $B_2'$ are independent
sets in $G$ but are cliques in $G^2$.  Let $H=G^2[B'_1\cup B'_2]$.  For each
$v\in V(H)$, let $L'(v) = L(v) \setminus \{c:\varphi(w) = c\mbox{ for some }w
\in N_{G^2}(v)\setminus (V(H)\cup D_2)\}$.  
Let $S_i=B_i'\cap N(N[N[S]]\cap D_2)$ for each $i\in \{1,2\}$.
Note that $S_i$ consists of vertices of $B_i'$ that are adjacent in $G^2$ to
(colored) vertices in $N[N[S]]$ via vertices in $D_2$.

To color $H$ by Lemma~\ref{lem:twocliques}, 
we first verify that each $v\in B'_i$ has at most 3 neighbors in $B'_{3-i}$
in $H$.  By symmetry, assume $v\in B'_1$.  Now each neighbor of $v$ in $B'_2$ in
$H$ is either adjacent to $v$ in $G$ or has a common neighbor with $v$ in $D_2$.
Further, each neighbor in $G$ in $V(R)$ yields at most one such neighbor in
$B'_2$, since $B'_2$ is independent in $G$ and each vertex in $D_2$ has degree
2.  So we are done by Lemma~\ref{lem:smalldegreeinregions}.

We must also verify that $S_1$ and $S_2$ are small enough and that $B_1',
B_2'$, and all of the lists $L'$ are big enough.  
Note that $\card{B_1\cap S}= 2$ and $\card{B_2\cap S}=2$, since each vertex of
$S$ must be on the first or last edge of the $r$-region, in $G'$, and each of
these edges has exactly one vertex in each of $B_1$ and $B_2$.  
Each $v\in N(S)\cap B_2$ has at most three neighbors, in $G^2$, in $B_1$ by
Lemma~\ref{lem:smalldegreeinregions}.  Further, one of these three is in $S$.
So $|S_1|=|B_1'\cap N(N[N[S]]\cap D_2)|\le 4|S\cap B_1|+2|S\cap B_2|\le
4(2)+2(2)=12$.  Similarly, $|S_2|\le 12$.

Consider $v\in B_1'\cup B_2'$; by symmetry, assume $v\in B_1'$.
Note that $|L'(v)| \ge \card{B_1'}$ whenever $v\in B_1'\setminus S_1$,
since each $v\in B'_1\setminus S_1$ loses at most one color for each vertex
in $(\{b_1,b_2\}\cup N(b_1))\setminus (V(H)\cup D_2)$, and
$D_2\cap(\{b_1,b_2\}\cup N(b_1))=\emptyset$.
Each $v\in B_1$ has at most three neighbors in $B_{2}$.  
Thus, each vertex $v\in S_1$ has at most three colored neighbors, in $G^2$, in
$B_2\setminus B'_2$.  So, $v$ loses at most three more colors than in the
analysis for vertices in $B'_1\setminus S_1$.  Hence, each $v\in S_1$ has
$|L'(v)|\ge \card{B_1'}-3$. 
Similarly, for each $v\in S_2$ we get $|L'(v)|\ge |B'_2|-3$.

Now we show that $B_1'$ and $B_2'$ are big enough. 
The number of edges of $G'$ incident with the region $R'$ is
$|R'|+1$. By 
Lemma~\ref{lem:smalldegreeinregions},
every vertex of $B_1$ or $B_2$ is
in at most three of those edges, so $|B_1|\ge (|R|+1)/3$ and $|B_2| \ge
(|R|+1)/3$; we can actually get better bounds using planarity, but we omit
that argument to keep the proof simpler.
Now $|S\cap B_1|=2$ and $|N(S)\cap B_1| \le
3|S\cap B_2| \le 6$, so $|N[S]\cap B_1|\le 8$. Hence $|B_1'|
\ge (|R|+1)/3 - |N[S]\cap B_1|\ge (161+1)/3-8=46$. 
Similary, $|B_2'| \ge 46$.

Thus, we can use Lemma~\ref{lem:twocliques} to extend the
coloring to $V(H)$.  After coloring $V(H)$, for each vertex $x\in D_2$, we can
color it arbitrarily from its list, since $|L(x)|\ge k+2$ and $d_{G^2}(x) \le
2\sqrt{k}$.  Hence, $G^2$ has an $L$-coloring, a contradiction.
\end{proof}

\begin{proof}[Proof of Theorem~\ref{th:main}]
Recall, from the start of Section~\ref{sec:proof}, that $G$ is a minimal
counterexample to Theorem~\ref{th:main}.  By Corollary~\ref{cor:highmultipB},
$G$ contains some $r$-region with $\frac{\sqrt{k}}{10} - 13 \le r$.
By Lemma~\ref{lem:smallregion}, $G$ contains no $r$-region with $r>160$.
Thus we have $\frac{\sqrt{k}}{10}-13\le 160$.  Simplifying gives
$k\le1,730^2=2,992,900$.  Thus, when $\Delta\ge 1,730^2+1$ we reach a
contradiction, which proves the theorem (our main result).  
\end{proof}
By relying more heavily on planarity, we can reduce the value of $\Delta_0$.
However, that approach adds numerous complications, which we prefer to avoid.

\section{A Coloring Algorithm and Extending to Paintability}
\label{paint}
In this section, we explain how the proof of Theorem~\ref{th:main} yields a
polynomial algorithm to color $G^2$ from its lists.  In fact, we give an
algorithm for the more general context of paintability (which we define below).
 Essentially, we construct a vertex order $\sigma$ such that we can consider
the vertices of $G$ in order $\sigma$ and color them greedily from their lists,
but there is a wrinkle.  If vertices appear together in an $r$-region, for
$r\ge 161$, then we consider them simultaneously, and color them as in the
proofs of Lemmas~\ref{lem:twocliques} and~\ref{lem:smallregion}.

Our proof of Theorem~\ref{th:main} in fact shows the following. 

\begin{thm}
\label{thm:struct}
Let $G$ be a planar graph with girth at least 5 and fix
$k\ge\max\{\Delta(G),1730^2+1\}$.
As in Section~\ref{sec:proof}, let $B$ denote the set of vertices $w$ with
$d(w)\ge \sqrt{k}$.  Now $G$ contains at least one of the following:
\begin{enumerate}
\item[(a)] only a single vertex,
\item[(b)] 2 or more components,
\item[(c)] a vertex of degree at most 1,
\item[(d)] an edge $uv$ such that $u\notin N[B]$ and $v\notin N[B]$,
\item[(e)] adjacent 2-vertices $u$ and $v$ such that $u\notin N(B)$ or $v\notin
N(B)$, and
\item[(f)] an $r$-region with $r\ge 161$.
\end{enumerate}
\end{thm}

\begin{proof}
The proof is essentially the same as the proof of
Corollary~\ref{cor:highmultipB}.  The key observation is that in that proof, and
the results upon which it depends, we do not explicitly use that $G$ is a
minimal counterexample to Theorem~\ref{th:main}; we only need that $G$ has no
instance of (a), (b), (c), (d), or (e).  So, as in
Corollary~\ref{cor:highmultipB}, we conclude that $G$ has an $r$-region with
$r\ge \frac{\sqrt{k}}{10}-13$.  Since $k>1730^2$, this gives $r>160$, as
desired.
\end{proof}

The game of \emph{$b$-paintability} (also called \emph{online
$b$-list-coloring}) is
played between two players, Lister and Painter.  On round $i$, Lister presents a
set $J_i$ of uncolored vertices.  Painter responds by choosing some independent
set $I_i\subseteq J_i$ to receive color $i$.  If Painter eventually colors
every vertex of the graph, then Painter wins.  If instead Lister presents some
uncolored vertex on $b$ rounds, but Painter never colors it, then Lister wins.
The \emph{paint number} $\chip(G)$ is the minimum $b$ such that Painter can win
regardless of how Lister plays.  Let $G$ be a planar graph with girth at least
five.  We show that if $G$ has maximum degree $\Delta\ge 1730^2+1=2,992,901$, then
$\chip(G^2)\le \Delta(G)+2$.

\begin{thm}
Let $G$ be a planar graph with girth at least 5.  Let
$k=\max\{\Delta(G),1730^2+1\}$.  Now $\chip(G)\le k+2$.
\end{thm}
\begin{proof}
A \emph{weak order} of a vertex set $V(G)$ is a generalization of a total
order, where we partition $V(G)$ into subsets and then form a total order on
these subsets.  In many cases, the subsets will be singletons, though not
always.  Further, each non-singleton subset gets a label that is all
vertices in the subset.  (Later, we may remove vertices from a subset, but we
never change its initial label.) For convenience, we simply list the subsets to
reflect the total order (from least to greatest).

Let $G$ satisfy the hypothesis.
We first construct a weak order $\sigma$ of $V(G)$, by induction on
$\card{V(G)}+\card{E(G)}$, using the six cases in
Theorem~\ref{thm:struct}, applying the first case that is applicable.

\begin{enumerate}
\item[(a)] $G=v$ and $\sigma=v$.

\item[(b)]
Suppose $G$ has 2 or more components $G_1,\ldots, G_t$, with $t\ge 2$.  By
hypothesis, construct weak orders $\sigma_1,\ldots,\sigma_t$, and form $\sigma$
by concatenating these, in any order.

\item[(c)]
If $G$ has a vertex $v$ of degree at most 1, then let $G'=G-v$.  Let $\sigma'$
be the weak order for $G'$ and form $\sigma$ by appending $v$ to $\sigma'$.

\item[(d)]
If $G$ has an edge $uv$ such that $u\notin N[B]$ and $v\notin N[B]$, then let
$G'=G-uv$.  Let $\sigma'$ be the order for $G'$.  Form $\sigma$ from $\sigma'$
by removing $u$ and $v$ from their places in the weak order and appending $u,v$.

\item[(e)]
If $G$ has adjacent 2-vertices $u$ and $v$ such that $u\notin N(B)$ or $v\notin
N(B)$, then by symmetry, assume $u\notin N(B)$.  Let $G'=G\setminus\{u,v\}$ and
let $\sigma'$ be the order for $G'$.  Form $\sigma$ from $\sigma'$ by appending
$v,u$.

\item[(f)]
If $G$ contains an $r$-region with $r\ge 161$, then define 
$B_1'$, $B_2'$, and $D_2$ as in the proof of Lemma~\ref{lem:smallregion}.  
Let $v_1$ and $v_2$ be vertices of $B_1'\cup B_2'\cup D_2$ such that every
vertex within distance 2 in $G$ of $v_1$ or $v_2$ is in $\{b_1,b_2\}\cup
N(b_1)\cup N(b_2)\cup V(R)$.  $G'=G-v_1v_2$ and let $\sigma'$ be the order for
$G'$.  Form $\sigma$ from $\sigma'$ by removing the vertices of $B_1',B_2'$, and
$D_2$ from wherever they appear, possibly in labeled (non-singleton) subsets,
appending the subset $B_1'\cup B_2'$ (with label $B_1'\cup B_2'$), followed
by each of the vertices of $D_2$ (as singletons) in arbitrary order.
\end{enumerate}

This completes the construction of the weak order $\sigma$ of $V(G)$.
Note that any labeled subset in $\sigma$ must arise from some $B_1'\cup
B_2'$ in (f).  Now we use $\sigma$ to describe a strategy for Painter to win
the $(k+2)$-painting game on $G^2$.

For a given round $i$, suppose Lister lists the set $J_i$.  Let $\sigma_i$
be the restriction to $J_i$ of $\sigma$.  Painter greedily constructs an
independent set $I_i$ as follows.  If the least element, $v$, in $\sigma_i$ is 
unlabeled, add it to $I_i$ and modify $\sigma_i$ by deleting all vertices
adjacent to $v$ in $G^2$.  Suppose instead the least element is a
labeled subset, call it $T_j$, which arose in (f) from some $r$-region, $R$.
Let $D$ be the digraph formed in the proof of Lemma~\ref{lem:twocliques} 
corresponding to $R$ (the vertices of $D$ are encoded in the label of $T_j$).
Since $D$ is kernel-perfect, $D[T_j]$ has a kernel, $T_j'$.  Now add the
vertices of $T_j'$ to $I_i$, and delete from $\sigma_i$ every vertex adjacent in
$G^2$ to one or more vertices of $T_j'$.  This completes the description of
Painter's strategy.  It can clearly be implemented in polynomial time.
Determining if an arbitrary graph has a kernel is NP-hard.  However, the proof
of Lemma~\ref{lem:twocliques} is constructive and gives rise to a simple
algorithm to find a kernel.

Finally, we show that Painter's strategy described above always wins the
$(k+2)$-painting game on $G^2$.  We need to consider vertices that were put into
$\sigma$ by each of (a) and (c)--(f). (In the process of recursively building $\sigma$,
a vertex $v$ may possibly be removed from a weak order for a smaller graph, and
reinserted at the end, as in (d) or (f).  In this case, we classify $v$
according to the final step that placed it in $\sigma$.)

\begin{enumerate}
\item[(a)] Suppose $v$ was put into $\sigma$ by (a).  Now $v$ has no
earlier neighbors (in $G^2$) in $\sigma$, so $v$ is colored on the first
round on which it appears.

\item[(c)]
Suppose vertex $v$ was put into $\sigma$ by
(c).  This means that $v$ has at most $k$ vertices that appear earlier in
$\sigma$ and are adjacent to $v$ in $G^2$.  Thus, $v$ can appear in
$J_i\setminus I_i$ on at most $k$ rounds.  So, when the game ends, $v$ is
colored.

\item[(d)]
Suppose vertex $v$ was put into $\sigma$ by (d).  Since $v\notin N[B]$, in
$G^2$ vertex $v$ has at most $(\sqrt{k})^2=k$ neighbors.  So, when the game
ends, $v$ is colored.

\item[(e)]
Suppose vertex $v$ was put into $\sigma$ by (e).  If the other 2-vertex $u$
put into $\sigma$ by (e) follows $v$, then at most $k+1$ vertices $w$ that are
adjacent in $G^2$ to $v$ precede $v$ in $\sigma$, so when the game ends $v$ will
be colored.  Otherwise, $v\notin N[B]$, so in $G^2$, vertex $v$ has at most $k$
neighbors, among vertices earlier in $\sigma$.  So, 
when the game ends, $v$ is colored.

\item[(f)]  Finally, suppose $v$ was put into $\sigma$ by (f).  If $v$ was in
$D_2$ for some $r$-region, then at most $k$ neighbors in $G^2$ of $v$ precede
$v$ in $\sigma$, so $v$ will be colored when the game ends.  Thus, we assume
$v\in B_1'\cup B_2'$ for some $r$-region (with $r\ge 161$);  by symmetry,
assume that $v\in B'_1$.  Let $D$ be the digraph formed in the proof of
Lemma~\ref{lem:twocliques} by orienting edges of $G^2[B_1'\cup B_2']$.  Recall
that $d^+_D(v)\le \card{B'_1}-1$ if $v\in B'_1\setminus S_1$.  Also,
$d^+_D(v)\le \card{B'_1}-4$ if $v\in S_1$.  First, suppose $v\in
B_1'\setminus S_1$.  By construction, $v$ has at most $k+2-\card{B_1'}$
neighbors in $G^2\setminus (B_1'\cup B_2'\cup D)$.  So, $v$ appears in
$J_i\setminus T_j$, due to neighbors in $G^2\setminus (B_1'\cup B_2'\cup D)$,
at most $k+2-\card{B_1'}$ times.  The number of times $v$ appears in
$T_j\setminus T_j'$ is at most $d^+_D(v)\le
\card{B_1'}-1$.  Thus, $v$ appears in $J_i\setminus T_j'$ at most
$(k+2-\card{B_1'})+\card{B_1'}-1=k+1$ times.  So $v$ is colored when the
game ends.  When $v\in S_1$, a similar analysis shows $v$ appears in
$J_i\setminus T_j$ at most $k+2-\card{B_1'}+4$ times and in
$T_j\setminus T_j'$ at most $14$ times.  Thus, $v$ appears in
$J_i\setminus T_j'$ at most $k+1$ times.  So 
when the game ends, $v$ is colored.
\end{enumerate}
This completes the proof that Painter wins the $(k+2)$-painting game on $G^2$.
\end{proof}

\section*{Acknowledgment}
Thanks to an anonymous referee, whose careful reading of earlier versions of
the paper caught numerous inaccuracies, as well as more serious mistakes in a
previous proof of Lemma~\ref{prop:40:stronger}.

\bibliographystyle{siam}
\footnotesize{
\bibliography{GraphColoring}
}

\end{document}